\RequirePackage[l2tabu, orthodox]{nag}
\documentclass[11pt,a4paper]{amsart}																										
\usepackage[includeheadfoot,paperwidth=210mm,paperheight=297mm,bottom=1.5cm,top=1.5cm,right=2cm,left=3cm,twoside]{geometry} 				
\usepackage{hyperref}
\hypersetup{	
	colorlinks,
	linkcolor=black,
	citecolor=black,
	urlcolor=black,
	breaklinks=true
}
\usepackage[applemac]{inputenc}
\usepackage[T1]{fontenc}
\usepackage[english]{babel}
\usepackage[textwidth=3cm]{todonotes}
\presetkeys{todonotes}{color=orange!85}{}

\usepackage{
	microtype,
	amssymb,
	mathtools,
}
\newtheoremstyle{lemma}{.5\baselineskip\@plus.2\baselineskip\@minus.2\baselineskip}{.5\baselineskip\@plus.2\baselineskip\@minus.2\baselineskip}
	{\itshape}
	{}
	{\bfseries}
	{.}
	{\newline}
	{\thmname{#1}\thmnumber{ #2}\thmnote{ (#3)}}	
\theoremstyle{lemma}
	\newtheorem{theorem}{Theorem}
	\newtheorem{lemma}[theorem]{Lemma}  
	\newtheorem{proposition}[theorem]{Proposition}
	\newtheorem{corollary}[theorem]{Corollary}  

\newtheoremstyle{definition}{.5\baselineskip\@plus.2\baselineskip\@minus.2\baselineskip}{.5\baselineskip\@plus.2\baselineskip\@minus.2\baselineskip}
	{}
	{}
	{\bfseries}
	{.}
	{\newline}
	{\thmname{#1}\thmnumber{ #2}\thmnote{ (#3)}}	
\theoremstyle{definition}
	\newtheorem{definition}[theorem]{Definition}

\def\XXint#1#2#3{{\setbox0=\hbox{$#1{#2#3}{\int}$} 
\vcenter{\hbox{$#2#3$}}\kern-.5\wd0}}

\makeatletter
\newcommand\avsuminner[2]{%
  {\sbox0{$\m@th#1\sum$}%
   \vphantom{\usebox0}%
   \ooalign{%
     \hidewidth
     \smash{\vrule height\dimexpr\ht0+1pt\relax depth\dimexpr\dp0+1pt\relax}%
     \hidewidth\cr
     $\m@th#1\sum$\cr
   }%
  }%
}
\makeatother

\DeclareMathOperator{\dist}{dist}

\DeclareMathOperator{\minRad}{minRad}
\DeclareMathOperator{\maxCurv}{maxCurv}
\DeclareMathOperator{\dcsd}{dcsd}

\DeclareMathOperator{\dcrit}{dcrit}

\newcommand{\operp}{\raisebox{0.25ex}{\ensuremath{\,{\scriptstyle\mathop{\bigcirc\kern-0.71em\perp}}\,}}}

\newcommand{\sphere}{\ensuremath{\mathbb{S}}}
\newcommand{\HM}{\ensuremath{\mathcal{H}}}

\newcommand{\M}{\ensuremath{\mathcal{M}}}

\DeclarePairedDelimiter\abs{\lvert}{\rvert}
\DeclarePairedDelimiter\norm{\lVert}{\rVert}

\newcommand{\dd}{\ensuremath{\mathrm{d}}}

\renewcommand{\phi}{\varphi}
\renewcommand{\epsilon}{\varepsilon}

\newcommand{\R}{\mathbb{R}}
\newcommand{\N}{\mathbb{N}}

\renewcommand{\phi}{\varphi}
\renewcommand{\epsilon}{\varepsilon}

\hypersetup{
	pdftitle={Discrete knot energies},
	pdfsubject={},
	pdfauthor={Sebastian Scholtes},
	pdfkeywords={},
	pdfcreator={},
	pdfproducer={}
}

\makeindex
\begin{document}

\title{Discrete knot energies}

\author{\href{mailto:sebastian.scholtes@rwth-aachen.de}{Sebastian Scholtes}}
\address{Lehrstuhl I f\"ur Mathematik\\RWTH Aachen University\\52062 Aachen\\Germany}
\email{\href{mailto:sebastian.scholtes@rwth-aachen.de}{sebastian.scholtes@rwth-aachen.de}}
\urladdr{\href{http://www.math1.rwth-aachen.de/~scholtes/}{http://www.math1.rwth-aachen.de/~scholtes/}}

\date{\today}
\begin{abstract}
	The present chapter gives an overview on results for discrete knot energies. 
	These discrete energies are designed to make swift numerical computations and thus open the field to computational methods. Additionally,
	they provide an independent, geometrically pleasing and consistent discrete model that behaves similarly to the original model. 
	We will focus on M\"obius energy, integral Menger curvature and thickness. 
\end{abstract}
\maketitle

\section{Introduction}

In classic knot theory, mathematicians are often interested in knot classes and how to distinguish between them, for example via \emph{knot invariants}. 
Contrary to this approach, geometric knot theory deals with the specific shape of knots and how to find or compute particularly nice representatives of a given knot class.
The exact meaning of nice depends on the context and can vary from applied (see the section on thickness) to theoretical considerations (see the section on integral Menger curvature).
To capture the ``quality '' of a knot, Fukuhara introduced the concept of a \emph{knot energy} (see \cite{Fukuhara1988a}).
An optimal representative is then said to be a minimizer of this energy among all curves of a given knot class.
Later on, this approach was further developed by other authors (see \cite{Simon1996a,Buck1997a,Sullivan2002a}) and by now, 
a functional on the space of knotted curves that is bounded from below and gets infinite as curves approach a self intersection is called a knot energy (see \cite{OHara2003a}).
This definition already includes some helpful ingredients to show that, indeed, minimizers of the energy exist.
The \emph{M\"obius energy}
\begin{align*}
	\mathcal{E}(\gamma)\vcentcolon=\int_{\sphere_{L}}\int_{\sphere_{L}}\left(\frac{1}{\abs{\gamma(t)-\gamma(s)}^{2}}-\frac{1}{d_{\sphere_{L}}(t,s)^{2}}\right)\,\dd s\,\dd t,
\end{align*}
for arc length curves $\gamma$, is a particular example of such an energy (see \cite{OHara1991a}).
In this chapter, we additionally consider \emph{integral Menger curvature} and \emph{thickness}. Each of these three energies is connected to curvature.
For an overview of regularizing and knot theoretic properties of different curvature energies see \cite{Strzelecki2013b,Strzelecki2013c,Strzelecki2014a,Blatt2014b,Blatt2014c}.
\bigskip

One way to think about knot invariants vs. knot energies is that
the first one gives rough information about the shape of the knot, for example, by restricting the knot to belong to a certain knot
class, while the energy resolves the finer details. For instance, an energy bound can in turn give bounds on the curvature, bi-Lipschitz constant, average crossing number or stick number.
\bigskip

The explicit shape of energy minimizing knots is only known in case of the unknot. For most energies, the energy optimal unknot is proven or expected to be the round circle (see \cite{Abrams2003a}).
Therefore, it is important to know how to make the energies accessible to computers and approximate such minimizers.
One approach is to investigate so-called \emph{discrete knot energies}, that is energies which are defined on polygonal knots. Here, the goal is to minimize these discrete energies in the class of polygonal knots
of a fixed knot class and fixed number of vertices and then prove that these minimizers converge to the minimizer of a ``smooth'' knot energy within the same knot class.
When trying to find such a discrete energy, the most obvious approach would be to restrict the original energy to the space of polygonal knots. 
However, such an approach does not work, as polygons have infinite energy.
This is related to the fact that all three energies are regularising, i.e. curves of finite energy are more regular than the curves for which the energy is well-defined. 
By borrowing appropriate concepts of curvature from discrete geometry and replacing smooth notions by discrete ones (for example integrals by sums), it is possible to define discrete energies 
in the same spirit as the original energies.
Hence, these energies are not merely discretisations but discrete versions of the original energy. 
A suitable convergence of energies, which to some extent already includes the convergence of minimizers, is \emph{$\Gamma$-convergence}.
For other modes of convergence see the Chapter \emph{Variational Convergence}.
\bigskip

For each of the three energies, there is a section in which we first consider the history of the energies and comment on recent developments.
Then we introduce the appropriate discrete energies and explain the connections to the original energies.
For the sake of clarity, we present most results for curves of length $1$. 
In the appendix, we give a short introduction to $\Gamma$-convergence.

\subsection{Notation}

In the following sections, $\mathcal{C}$ is the space of closed arc length curves with length $1$ and $\mathcal{P}_{n}$ the subspace of equilateral polygons with $n$ segments. 
Furthermore, we abbreviate $\mathcal{C}_{k,\mathrm{p}}\vcentcolon=\mathcal{C}\cap (C^{k}\cup \bigcup_{n\in\N}\mathcal{P}_{n})$, where $C^{k}$ is the class of $k$ times continuously differentiable functions.
We write $C^{k,\alpha}$ for the class of functions in $C^{k}$ whose derivatives are H\"older continuous with exponent $\alpha$ and 
$L^{q}$ for the Lebesgue spaces of $q$-integrable functions.
By $W^{k,q}$ we denote the standard Sobolev spaces of $k$-times weakly differentiable closed curves with $q$-integrable weak derivative.
Adding a knot class $\mathcal{K}$ in brackets to a set of curves restricts this set to the subset of curves that belong to the knot class $\mathcal{K}$.
The circle of length $L>0$ is denoted by $\sphere_{L}$.

\section{M\"obius Energy}

The \emph{M\"obius energy}
\begin{align*}
	\mathcal{E}(\gamma)\vcentcolon=\int_{\sphere_{L}}\int_{\sphere_{L}}\left(\frac{1}{\abs{\gamma(t)-\gamma(s)}^{2}}-\frac{1}{d_{\sphere_{L}}(t,s)^{2}}\right)\,\dd s\,\dd t,
\end{align*}
is defined on closed rectifiable arc length curves $\gamma$ of length $L$. Here, $d_{\sphere_{L}}$ is the intrinsic metric on $\sphere_{L}$. 
This energy was introduced by O'Hara (see \cite{OHara1991a}) and has the interesting property that it is invariant under M\"obius transformations, hence its name. 
O'Hara could show that finite energy prevents the curve from having self intersections. 
Later on, the existence of energy minimizers in
prime knot classes was proven by Freedman, He and Wang (see \cite{Freedman1994a}), while there is a folklore conjecture, 
usually attributed to Kusner and Sullivan, questioning the existence in composite knot classes based on computer experiments (see \cite{Kusner1997a}).
Additionally, it was shown that the unique absolute minimizer is the round circle (see \cite{Freedman1994a}).
Similar uniqueness results are known for broader classes of energies (see \cite{Abrams2003a}).
The regularity of minimizers and, more generally, of critical points was investigated and smoothness could be proven (see \cite{Freedman1994a,He2000a,Reiter2009a,Reiter2010a,Blatt2015b}).
Furthermore, it was shown that the M\"obius energy of a curve is finite if and only if the curve is simple and the arc length parametrisation belongs to the fractional Sobolev space 
$\cramped{W^{3/2,2}}$ (see \cite{Blatt2012a}). The gradient flow of the M\"obius energy was investigated (see \cite{Blatt2012e,Blatt2016a}) and results for the larger class of O'Hara's knot energies (see \cite{OHara1992a,OHara1992b,OHara1994a})
are available (see \cite{Blatt2008b,Reiter2012a,Blatt2013b,Blatt2016b}).
\bigskip

A discrete version of the M\"obius energy, called \emph{minimum distance energy},
was introduced by Simon (see \cite{Simon1994b}). If $p$ is a polygon with $n$ consecutive segments
$X_{i}$ this energy is defined by
\begin{align}\label{MDenergy}
	\begin{split}
		\mathcal{E}_{\mathrm{md},n}(p)\vcentcolon=\mathcal{U}_{\text{md},n}(p)-\mathcal{U}_{\text{md},n}(g_{n})\quad\text{with}\quad
		\mathcal{U}_{\mathrm{md},n}(p)\vcentcolon=\sum_{i=1}^{n}\sum_{\substack{j=1\\X_{i},X_{j}\text{ not adjacent}}}^{n}\frac{\abs{X_{i}}\abs{X_{j}}}{\dist(X_{i},X_{j})^{2}},
	\end{split}
\end{align}
where $g_{n}$ is the regular $n$-gon. Note, that this energy is scale invariant.
Similar energy functionals for polygons were considered previously (see \cite{Fukuhara1988a,Buck1993a}).
There is also a variant of the M\"obius energy for graphs (see \cite{Karpenkov2006a}) and
for curves with self intersections (see \cite{Dunning2011a}).
It is know that the minimum distance energy is Lipschitz continuous on sublevel sets:

\begin{theorem}[Minimum distance energy is Lipschitz continuous, \cite{Simon1994a}]
	The minimum distance energy $\mathcal{E}_{\mathrm{md},n}$ is continuous on the space of nonsingular polygons with $n$ segments. 
	Moreover, it is Lipschitz continuous on the subspace of polygons of length at least $\ell$, whose energy is bounded by a constant $E$.
\end{theorem}

Furthermore, the existence of minimizers for every tame knot class was established:

\begin{theorem}[Existence of discrete minimizers for $\mathcal{E}_{\mathrm{md},n}$, \cite{Simon1994a}]
	For every tame knot class $\mathcal{K}$ there is a minimizer of $\mathcal{E}_{\mathrm{md},n}$ in the set of all polygons with $n$ vertices that belong to $\mathcal{K}$.
\end{theorem}

For this energy, approximation results for suitably inscribed polygons could be shown:

\begin{theorem}[Explicit energy bound and convergence for inscribed polygons, \cite{Rawdon2006a}]
	Let $\gamma\in \mathcal{C}\cap C^{2}$ and let $p_{n}$ be inscribed polygons that divide $\gamma$ in $n$ arcs of length $\frac{1}{n}$. Then
	\begin{align*}
		\abs*{\mathcal{E}(\gamma)-\mathcal{E}_{\mathrm{md},n}(p_{n})}\leq \frac{C(\gamma)}{n^{\frac{1}{4}}}.
	\end{align*}
	If $n$ is large enough, the constant $C(\gamma)$ can be chosen as
	\begin{align*}
		C(\gamma)=\frac{290}{\Delta[\gamma]^{\frac{1}{4}}},
	\end{align*}
	where $\Delta[\gamma]$ is the thickness of $\gamma$ (see Section \ref{sectionthickness}). Hence,
	\begin{align*}
		\mathcal{E}_{\mathrm{md},n}(p_{n})\to\mathcal{E}(\gamma).
	\end{align*} 
\end{theorem}

Furthermore, an explicit error bound on the difference between the minimum distance energy of an equilateral polygonal knot
and the M\"obius energy of a smooth knot, appropriately inscribed in the polygonal knot, could be established in terms of thickness and the number of segments:

\begin{theorem}[Explicit energy bound for inscribed smooth knots, \cite{Rawdon2010a}]
	Let $p$ be an equilateral polygon of length $1$. Then there is an ``inscribed'' $C^{2}$ knot $\gamma_{p}$ such that 
	\begin{align*}
		\abs*{\mathcal{E}_{\mathrm{md},n}(p)-\mathcal{E}(\gamma_{p})}
		\leq \frac{C_{1}(p)}{n^{\frac{1}{4}}}+\frac{C_{2}(p)}{n}+\frac{C_{3}(p)}{n^{\frac{5}{4}}}+\frac{C_{4}(p)}{n^{\frac{7}{4}}}+\frac{C_{5}(p)}{n^{2}}
	\end{align*}
	and the constants $C_{i}(p)$ depend in an explicit way on negative powers of $\Delta_{n}[p]$.
\end{theorem}

However, from these results it is not possible to infer that the minimal minimum distance
energy converges to the minimal M\"obius energy in a fixed knot class. 
For the overall minimizers of the minimum distance energy the following result is known:

\begin{theorem}[Minimizers of the minimum distance energy, \cite{Tam2006a,Speller2007a,Speller2008a}]
	The minimizers of the minimum distance energy $\mathcal{E}_{\mathrm{md},n}$ are convex and for $n\in\{4,5\}$ these minimizers are the regular $n$-gon.
\end{theorem}

This evidence supports the conjecture that the regular $n$-gon minimizes the minimum distance energy in the class of $n$-gons.
Numerical experiments regarding the minimum distance energy under the elastic flow were carried out (see \cite{Hermes2008a}).
\bigskip

Another, more obvious, discrete version of the M\"obius energy was used for numerical experiments (see \cite{Kim1993a}).
This energy, defined on the class of arc length parametrisations of polygons of length $L$ with $n$ segments, is given by
\begin{align}\label{discreteMoebiusenergy}
	\mathcal{E}_{n}(p)\vcentcolon=\sum_{\substack{i,j=1\\i\not=j}}^{n}\left(\frac{1}{\abs{p(a_{j})-p(a_{i})}^{2}}-\frac{1}{d(a_{j},a_{i})^{2}}\right)d(a_{i+1},a_{i})d(a_{j+1},a_{j}),
\end{align}
where the $a_{i}$ are consecutive points on $\sphere_{L}$ and $p(a_{i})$ the vertices of the polygon.
Also this energy is scale invariant and it is easily seen to be continuous on the space of nonsingular polygons with $n$ segments.
The relationship between this discrete M\"obius energy and the classic M\"obius energy could be established in terms of a $\Gamma$-convergence result:

\begin{theorem}[M\"obius energy is $\Gamma$-limit of discrete M\"obius energies, \cite{Scholtes2014e}]\label{maintheoremgammamoebius}
	For $q\in [1,\infty]$, $\norm{\cdot}\in\{\norm{\cdot}_{L^{q}(\sphere_{1},\R^{d})},\norm{\cdot}_{W^{1,q}(\sphere_{1},\R^{d})}\}$ and every tame knot class $\mathcal{K}$ holds 
	\begin{align*}
		\mathcal{E}_{n}\xrightarrow{\Gamma} \mathcal{E}\quad\text{on }\left(\mathcal{C}_{1,\mathrm{p}}(\mathcal{K}),\norm{\cdot}\right).
	\end{align*}
\end{theorem}

Since we already know that minimizers of $\mathcal{E}$ in prime knot classes exist and are smooth, Theorem \ref{dalmasoproposition} implies the convergence of discrete almost minimizers:

\begin{corollary}[Convergence of discrete almost minimizers, \cite{Scholtes2014e}]\label{corollaryconvergenceminmizers}
	Let $\mathcal{K}$ be a tame prime knot class, $p_{n}\in\mathcal{P}_{n}(\mathcal{K})$ with
	\begin{align*}
		\Bigl|\inf_{\mathcal{P}_{n}(\mathcal{K})}\mathcal{E}_{n}-\mathcal{E}_{n}(p_{n})\Bigr|\to 0\qquad\text{and}\qquad
		p_{n}\to\gamma\in\mathcal{C}(\mathcal{K})\text{ in }L^{1}(\sphere_{1},\R^{d}).
	\end{align*}
	Then $\gamma$ is a minimizer of $\mathcal{E}$ in $\mathcal{C}(\mathcal{K})$ and $\lim_{k\to\infty}\mathcal{E}_{n}(p_{n})=\mathcal{E}(\gamma)$.
\end{corollary}

The result remains true for subsequences, where the number of edges is allowed to increase by more than $1$ for two consecutive polygons.
Since all curves are parametrised by arc length, it is not hard to find a
subsequence of the almost minimizers that converges in $C^{0}$, but generally this does not guarantee that the limit curve belongs to the same knot class or is parametrised by arc length.
For polygons inscribed in a $C^{1,1}$ curve, there is an estimate on the order of convergence:

\begin{proposition}[Order of convergence for M\"obius energy, \cite{Scholtes2014e}]\label{orderofconvergence}
	Let $\gamma\in C^{1,1}(\sphere_{L},\R^{d})$ be parametrised by arc length and $c,\overline c>0$. Then for every $\epsilon\in (0,1)$ there is a constant $C_{\epsilon}>0$ such that 
	\begin{align*}
		\abs{\mathcal{E}(\gamma)-\mathcal{E}_{n}(p_{n})}\leq \frac{C_{\epsilon}}{n^{1-\epsilon}}
	\end{align*}
	for every inscribed polygon $p_{n}$ given by a subdivision $b_{k}$, $k=1,\ldots,n$ of $\sphere_{L}$ such that 
	\begin{align*}
		\frac{c}{n}\leq \min_{k=1,\ldots,n}\abs{\gamma(b_{k+1})-\gamma(b_{k})} \leq \max_{k=1,\ldots,n}\abs{\gamma(b_{k+1})-\gamma(b_{k})}\leq \frac{\overline c}{n}.
	\end{align*}
\end{proposition}

This is in accordance with the data from computer experiments, which suggests that the order of convergence should be roughly $1$ (see \cite{Kim1993a}). 
If no regularity is assumed, the order of convergence might not be under control, but still, the energies converge:

\begin{corollary}[Convergence of M\"obius energies of inscribed polygons, \cite{Scholtes2014e}]\label{convergencemoebiusinscribedpolygons}
	Let $\gamma\in\mathcal{C}$ with $\mathcal{E}(\gamma)<\infty$ and $p_{n}$ as in Proposition \ref{orderofconvergence}. Then $\lim_{n\to\infty}\mathcal{E}_{n}(p_{n})=\mathcal{E}(\gamma)$.
\end{corollary}

In contrast to the situation for the minimum distance energy, the overall minimizers of the discrete M\"obius energy are known:

\begin{lemma}[Regular $n$-gon is unique minimizer of discrete M\"obius energy, \cite{Scholtes2014e}]\label{discreteminimizermoebius}
	The unique minimizer of $\mathcal{E}_{n}$ in $\mathcal{P}_{n}$ is the regular $n$-gon.
\end{lemma}

An immediate consequence is the convergence of overall discrete minimizers to the round circle:

\begin{corollary}[Convergence of discrete minimizers to the round circle, \cite{Scholtes2014e}]\label{convergenceofdiscreteminimizerstoroundcircle}
	Let $p_{n}\in\mathcal{P}_{n}$ bounded in $L^{\infty}$ with $\mathcal{E}_{n}(p_{n})=\inf_{\mathcal{P}_{n}}\mathcal{E}_{n}$.
	Then there is a subsequence with $p_{n_{k}}\to \gamma$ in $W^{1,\infty}(\sphere_{1},\R^{d})$, where $\gamma$ is a round unit circle.
\end{corollary}

One of the main differences between the discrete M\"obius energy \eqref{discreteMoebiusenergy} and the minimum distance energy \eqref{MDenergy} is that bounded minimum distance energy avoids double point singularities, 
while for \eqref{discreteMoebiusenergy} this is only true in the limit. This avoidance of singularities permits to prove the existence of minimizers of the minimum distance energy
\eqref{MDenergy} via the direct method. This might be harder or even impossible to achieve for the energy \eqref{discreteMoebiusenergy}. Nevertheless, the relation between the discrete 
M\"obius energy \eqref{discreteMoebiusenergy} and the smooth M\"obius energy is more clearly visible than for the minimum distance energy \eqref{MDenergy}, as reflected in Theorem \ref{maintheoremgammamoebius} and 
Corollaries \ref{corollaryconvergenceminmizers}-\ref{convergenceofdiscreteminimizerstoroundcircle}.

\section{Integral Menger Curvature}

The \emph{integral Menger curvature} was first considered by Mel'nikov (see \cite{Melnikov1995b}) as a concept of curvature of a measure that is naturally connected with the Cauchy transform
of this measure. For closed arc length curves $\gamma$ of length $L$, the integral Menger curvature is given by
\begin{align*}
	\mathcal{M}_{s}(\gamma)&\vcentcolon=\int_{\sphere_{L}}\int_{\sphere_{L}}\int_{\sphere_{L}}\kappa^{s}\bigl(\gamma(t),\gamma(u),\gamma(v)\bigr)\,\dd t\,\dd u\,\dd v,
\end{align*}
where $s\in (0,\infty)$ and $\kappa(x,y,z)$ is the inverse of the circumradius $r(x,y,z)$ of the three points $x,y$ and $z$.
The integral Menger curvature for $s=2$ played an important role in the solution of the Painlev\'e problem, i.e. to find geometric characterisations of removable sets for 
bounded analytic functions, see \cite{Pajot2002a,Dudziak2010a,Tolsa2014a} for a detailed presentation and references.
There is a remarkable theorem, which states that one\--di\-men\-sio\-nal Borel sets in $\R^{d}$ with finite integral Menger curvature $\mathcal{M}_{2}$ are $1$-rectifiable (see \cite{Leger1999a}). 
These results for $\M_{2}$ were later extended to sets of fractional dimension and metric spaces (see \cite{Lin2001a,Hahlomaa2008a}).
As a consequence, this theorem also ensures that an $\HM^{1}$ measurable set $E\subset \R^{d}$ with $\mathcal{M}_{2}(E)<\infty$ has approximate $1$-tangents at $\HM^{1}$ a.e point.
\bigskip

Complementary to this research, where highly irregular sets are permitted, is the investigation of rectifiable curves with finite $\M_{s}$ energy.
These curves have a classic tangent $\HM^{1}$ a.e. to begin with.
It turns out that for $s>3$ this guarantees that the curve is simple and that the arc length parametrisation is of class $\cramped{C^{1,1-3/p}}$,
which can be interpreted as a geometric Morrey-Sobolev imbedding (see \cite{Strzelecki2010a}).
It could be shown that the space of curves with finite $\mathcal{M}_{s}$ for $s>3$ is that of Sobolev-Slobodeckij 
embeddings of class $\cramped{W^{2-2/s,s}}$ (see \cite{Blatt2013a}) and that polygons have finite integral Menger curvature $\mathcal{M}_{s}$ exactly for $s\in (0, 3)$ (see \cite{Scholtes2011e}).
Furthermore, results regarding optimal H\"older regularity could be obtained (see \cite{Kolasinski2013a}).
Related energies have been investigated with regard to their regularizing properties (see \cite{Strzelecki2007a,Strzelecki2009a,Blatt2015a}).
\bigskip

The movement of quadrilaterals according to their Menger curvature was investigated (see \cite{Jecko2002a}) and
computer experiments for the gradient flow of integral Menger curvature were carried out (see \cite{Hermes2012a}). 
Besides the ad hoc method used there, the only theoretical results
regarding discrete versions of the integral Menger curvature that we are aware of can be found in the thesis of the author (see \cite{Scholtes2014a}). 
These results are still subject to ongoing research and are soon to be extended in an article of the author (see \cite{Scholtes2016b}).
There, the \emph{discrete integral Menger curvature} of a polygon $p$ with consecutive vertices $p(a_{i})=x_{i}$ is defined by
\begin{align*}
	\mathcal{M}_{s,n}(p)\vcentcolon=\sum_{\substack{i,j,k=1\\\#\{i,j,k\}=3}}^{n}\kappa^{s}(x_{i},x_{j},x_{k})\prod_{l\in\{i,j,k\}}\frac{\abs{x_{l}-x_{l-1}}+\abs{x_{l+1}-x_{l}}}{2}
\end{align*}
if $x_{i}\not=x_{j}$ for $i\not=j$ and $\mathcal{M}_{s,n}(p)=\infty$ else.
It is easily seen, that this energy is continuous on the space of nonsingular polygons with $n$ segments.
As for the M\"obius energy, there is a $\Gamma$ convergence result:

\begin{theorem}[Integral Menger curvature is $\Gamma$-limit of discrete energies, \cite{Scholtes2014a}]\label{maintheoremgammamenger}
	For $q\in [1,\infty]$, $\norm{\cdot}\in\{\norm{\cdot}_{L^{q}(\sphere_{1},\R^{d})},\norm{\cdot}_{W^{1,q}(\sphere_{1},\R^{d})}\}$ and every tame knot class $\mathcal{K}$ holds 
	\begin{align*}
		\mathcal{M}_{s,n}\xrightarrow{\Gamma}\mathcal{M}_{s}\quad\text{on }(\mathcal{C}_{2,\mathrm{p}}(\mathcal{K}),\norm{\cdot}).
	\end{align*}
\end{theorem}

Additionally, it could be shown that the energies of inscribed polygons converge to the energy of the curve if the number of vertices increases:

\begin{corollary}[Convergence of Menger curvatures of inscribed polygons, \cite{Scholtes2014a}]\label{energyconvergenceinscribedpolygonsmenger}
	Let $s\in (0,\infty)$, $\gamma\in \mathcal{C}\cap C^{2}$ embedded and $p_{n}$ be inscribed equilateral polygons with $n$ segments. 
	Then $\lim_{n\to\infty}\mathcal{M}_{s,n}(p_{n})=\mathcal{M}_{s}(\gamma)$.
\end{corollary}

\section{Thickness}\label{sectionthickness}

The \emph{thickness} $\Delta[\gamma]$ of a curve $\gamma$ was introduced by Gonzalez and Maddocks (see \cite{Gonzalez1999a}) as
\begin{align*}
	\Delta[\gamma]\vcentcolon=\inf_{s\not= t\not= u\not= s}r\bigl(\gamma(s),\gamma(t),\gamma(u)\bigr)
\end{align*}
and is equivalent to Federer's reach (see \cite{Federer1959a}). 
As in the previous section, $r(x,y,z)$ is the circumradius of the three points $x,y$ and $z$.
Geometrically, the thickness of a curve gives the radius of the largest uniform tubular neighbourhood about the curve that does not intersect itself. 
The \emph{ropelength}, which is length divided by thickness, is scale invariant and a knot is called \emph{ideal} if it minimizes ropelength in 
a fixed knot class or, equivalently, minimizes this energy amongst all curves in this knot class with fixed length. Ideal knots are of great interest, not only to mathematicians but also to biologists, 
chemists and physicists, since they exhibit interesting physical features and resemble the time-averaged shapes of knotted DNA molecules in solution (see \cite{Stasiak1996a,Katritch1996a,Katritch1997a}
and \cite{1998a,Simon2002a} for an overview of physical knot theory with applications). 
The existence of ideal knots in every knot class was settled by different teams of authors (see \cite{Cantarella2002a,Gonzalez2002b,Gonzalez2003a}) and it was found that the unique absolute minimizer
is the round circle.
Furthermore, this energy is self-repulsive, meaning that finite energy prevents the curve from having self intersections. By now it is well-known that thick curves, or in general manifolds of
positive reach, are of class $C^{1,1}$ and vice versa (see \cite{Lucas1957a,Federer1959a,Schuricht2003a,Lytchak2005a,Scholtes2013a}).
It was shown that ideal links must not be of class $C^{2}$ (see \cite{Cantarella2002a}) and computer experiments suggest that
$C^{1,1}$ regularity is optimal for knots, too (see \cite{Sullivan2002a}).
Further computer experiments were carried out with the software packages \texttt{SONO}, \texttt{libbiarc} and \texttt{ridgerunner} (see \cite{Pieranski1998a,Carlen2010a,Ashton2011a}).
A previous conjecture (see \cite[Conjecture 24]{Cantarella2002a}) that ropelength minimizers are piecewise analytic seems to be reversed by numerical results,
which indicate that there might be more singularities than previously expected (see \cite{Baranska2008a,Przybyl2014b}). 
Further interesting properties of critical points as well as the Euler-Lagrange equation were investigated (see \cite{Schuricht2003a,Schuricht2004a,Cantarella2014b}).
\bigskip

Another way to write the thickness of a thick arc length curve is 
\begin{align}\label{formulathickness2}
	\Delta[\gamma]=\min\left\{\minRad(\gamma),2^{-1}\dcsd(\gamma)\right\}
\end{align} 
(see \cite[Theorem 1]{Litherland1999a}). 
The minimal radius of curvature $\minRad(\gamma)$ of $\gamma$ is the inverse of the maximal curvature 
\begin{align*}
	\maxCurv(\gamma)\vcentcolon=||\kappa||_{L^{\infty}}\qquad\text{and}\qquad\dcsd(\gamma)\vcentcolon=\min_{(x,y)\in \dcrit(\gamma)}\abs{y-x} 
\end{align*}
is the doubly critical self distance. The set of doubly critical points $\dcrit(\gamma)$ of a $C^{1}$ curve $\gamma$ consists of all
pairs $(x,y)$ where $x=\gamma(t)$ and $y=\gamma(s)$ are distinct points on $\gamma$ so that 
\begin{align*}
	\langle \gamma'(t),\gamma(t)-\gamma(s) \rangle=\langle \gamma'(s),\gamma(t)-\gamma(s) \rangle=0,
\end{align*}
i.e. $s$ is critical for $u\mapsto\abs{\gamma(t)-\gamma(u)}^{2}$ and $t$ for $v\mapsto\abs{\gamma(v)-\gamma(s)}^{2}$.
\bigskip

The \emph{discrete thickness} $\Delta_{n}$, derived from the representation in \eqref{formulathickness2}, was introduced by Rawdon (see \cite{Rawdon1997a}).
The curvature of a polygon, localized at a vertex $y$,
is defined by
\begin{align*}
	\kappa_{d}(x,y,z)\vcentcolon=\frac{2\tan(\frac{\phi}{2})}{\frac{\abs{x-y}+\abs{z-y}}{2}},
\end{align*}
where $x$ and $z$ are the vertices adjacent to $y$ and $\phi=\measuredangle(y-x,z-y)$ is the exterior angle at $y$.
We then set 
\begin{align*}
	\minRad(p)\vcentcolon=\maxCurv(p)^{-1}\vcentcolon=\min_{i=1,\ldots,n}\kappa_{d}^{-1}(x_{i-1},x_{i},x_{i+1})
\end{align*}
if the polygon $p$ has the consecutive vertices $x_{i}$, $x_{0}\vcentcolon= x_{n}$, $x_{n+1}\vcentcolon=x_{1}$.
The doubly critical self distance of a polygon $p$ is given as for a smooth curve if we define $\dcrit(p)$ to consist of 
pairs $(x,y)$ where $x=p(t)$ and $y=p(s)$ and  $s$ locally extremizes $u\mapsto\abs{p(t)-p(u)}^{2}$ and $t$ locally extremizes $v\mapsto\abs{p(v)-p(s)}^{2}$.
Now, $\Delta_{n}$ is defined analogous to \eqref{formulathickness2} by
\begin{align*}
	\Delta_{n}[p]=\min\left\{\minRad(p),2^{-1}\dcsd(p)\right\}
\end{align*}
if all vertices are distinct and $\Delta_{n}[p]=0$ if two vertices of $p$ coincide.
In a series of works (see \cite{Rawdon1997a,Rawdon1998a,Rawdon2000a,Rawdon2003a,Millett2008a}) alternative representations and properties of the discrete thickness were established.
For example, it was shown that the discrete thickness is continuous:

\begin{theorem}[Discrete thickness is continuous, \cite{Rawdon1997a}]
	The discrete thickness $\Delta_{n}$ is continuous on the space of simple closed polygons with $n$-segments with regard to the metric
	\begin{align*}
		d(p,q)\vcentcolon=\sum_{k=1}^{n}\abs{x_{i}-y_{i}},
	\end{align*}
	where $p$ and $q$ are polygons with vertices $x_{i}$ and $y_{i}$.
\end{theorem}

Another important result is the existence of discrete minimizers:

\begin{theorem}[Existence of discrete minimizers for $\Delta_{n}$, \cite{Rawdon2003a}]
	For every tame knot class $\mathcal{K}$ there is an ideal polygonal knot in $\mathcal{P}_{n}(\mathcal{K})$. 
\end{theorem}

More generally, the result remains true if instead of equilateral polygons one takes the space of polygons with a uniform bound on longest to shortest segment length.
These results were then used to prove the convergence of ideal polygonal to smooth ideal knots, a result that could later be improved from $C^{0}$ convergence to $C^{0,1}$ convergence

\begin{corollary}[Ideal polygonal knots converge to smooth ideal knots, \cite{Rawdon2003a,Scholtes2014d}]\label{theoremconvergenceofminimizers}
	Let $\mathcal{K}$ be a tame knot class and $p_{n}\in\mathcal{P}_{n}(\mathcal{K})$ bounded in $L^{\infty}$ with $\abs{\inf_{\mathcal{P}_{n}(\mathcal{K})}\Delta_{n}^{-1}-\Delta_{n}[p_{n}]^{-1}}\to 0$. 
	Then, there is a subsequence
	\begin{align*}
		p_{n_{k}}\xrightarrow[k\to\infty]{W^{1,\infty}(\sphere_{1},\R^{3})}\gamma\in \mathcal{C}(\mathcal{K})\qquad\text{with}\qquad
		\Delta^{-1}[\gamma]=\inf_{\mathcal{C}(\mathcal{K})}\Delta^{-1}=\lim_{k\to\infty}\Delta_{n_{k}}^{-1}[p_{n_{k}}].
	\end{align*}
\end{corollary}

The relationship on a functional level is again captured by a $\Gamma$-convergence result:

\begin{theorem}[Convergence of discrete inverse to smooth inverse thickness, \cite{Scholtes2014d}]\label{gammaconvergence}
	For every tame knot class $\mathcal{K}$ holds
	\begin{align*}
		\Delta_{n}^{-1}\xrightarrow{\Gamma}\Delta^{-1}\quad\text{on }(\mathcal{C}(\mathcal{K}),||\cdot||_{W^{1,\infty}(\sphere_{1},\R^{3})}).
	\end{align*}
\end{theorem}

Similar questions for more general energies were considered (see \cite{Dai2000a,Rawdon2003a}).
If the knot class is not fixed, the unique absolute minimizers of $\Delta_{n}^{-1}$ is the regular $n$-gon:

\begin{proposition}[Regular $n$-gon is unique minimizer of $\Delta_{n}^{-1}$, \cite{Scholtes2014d}]\label{regularngonminimizesdiscretethickness}
	The unique minimizer of $\Delta_{n}^{-1}$ in $\mathcal{P}_{n}$ is the regular $n$-gon.
\end{proposition}

\begin{appendix}

\section{Postlude in $\Gamma$-convergence}

In this section, we repeat some relevant facts on $\Gamma$-convergence. For more details, we refer to the books by Dal Maso and Braides (see \cite{Dal-Maso1993a,Braides2002a}).
This notion of convergence for functionals was introduced by DeGiorgi and is devised in a way, 
as to allow the convergence of minimizers and even almost minimizers.

\begin{definition}[$\,\Gamma$-convergence]
	Let $X$ be a topological space, $\mathcal{F},\mathcal{F}_{n}:X\to\overline \R\vcentcolon=\R\cup\{\pm\infty\}$. 
	Then \emph{$\mathcal{F}_{n}$ $\Gamma$-converges to $\mathcal{F}$}, if
	\begin{itemize}
		\item
			for every $x_{n}\to x$ holds $\mathcal{F}(x)\leq \liminf_{n\to\infty}\mathcal{F}_{n}(x_{n}$),
		\item
			for every $x\in X$ there are $x_{n}\to x$ with $\limsup_{n\to\infty}\mathcal{F}_{n}(x_{n})\leq \mathcal{F}(x)$.
	\end{itemize}
\end{definition}

The first inequality is usually called $\liminf$ inequality and the second one $\limsup$ inequality. If the functionals are only defined on subspaces $Y$ and $Y_{n}$
of $X$ and we extend the functionals by plus infinity on the rest of $X$, it is enough to show that the $\liminf$ inequality holds for every $x_{n}\in Y_{n}$, $x\in X$ 
and the $\limsup$ inequality for $x\in Y$ and $x_{n}\in Y_{n}$ to establish $\Gamma$-convergence. 
We want to use $\Gamma$-convergence to ensure that minimizers of the discrete functional $\mathcal{F}_{n}$ converge to minimizers of the ``smooth'' functional $\mathcal{F}$.

\begin{theorem}[Convergence of minimizers, {\cite[Corollary 7.17, p.78]{Dal-Maso1993a}}]\label{dalmasoproposition}
	Let $\mathcal{F}_{n},\mathcal{F}:X\to\overline \R$ with $\mathcal{F}_{n}\stackrel{\Gamma}{\to}\mathcal{F}$. Let $\epsilon_{n}>0$, $\epsilon_{n}\to 0$ and $x_{n}\in X$ 
	with $|\inf \mathcal{F}_{n}-\mathcal{F}_{n}(x_{n})|\leq \epsilon_{n}$. If $x_{n_{k}}\to x$, then
	\begin{align*}
		\mathcal{F}(x)=\inf \mathcal{F}=\lim_{k\to\infty}\mathcal{F}_{n}(x_{n_{k}}).
	\end{align*}
\end{theorem}

Note, that the previous theorem does not imply convergence of a subsequence, but instead assumes that we already have such a sequence to begin with. 
This fact is often taken care of by an accompanying compactness result.
\bigskip

We prove that the integral Menger curvature converges to the inverse thickness, to acquaint the reader with a particularly simple case of such a convergence theorem:

\begin{lemma}[Convergence of integral Menger curvature to inverse thickness]
	The integral Menger curvature converges to inverse thickness
	\begin{align*}
		\mathcal{M}_{s}^{\frac{1}{s}}\xrightarrow{\Gamma}\Delta^{-1}\quad\text{on }(\mathcal{C}(\mathcal{K}),\norm{\cdot}_{L^{\infty}}).
	\end{align*}
	Let $\mathcal{K}$ be a tame knot class and let $s\in (3,\infty)$.
	Then, there are minimizers $\gamma_{s}$ of $\mathcal{M}_{s}$ in $\mathcal{C}(\mathcal{K})$. Moreover, we have that  
	\begin{align*}
		\inf_{\mathcal{C}(\mathcal{K})}\mathcal{M}_{s}^{1/s}\xrightarrow[s\to\infty]{}\inf_{\mathcal{C}(\mathcal{K})}\Delta^{-1}
	\end{align*}
	and, after translating the minimizers if necessary, there is subsequence such that
	\begin{align*}
		\gamma_{s_{k}}\xrightarrow[k\to\infty]{C^{1}}\gamma\in \mathcal{C}(\mathcal{K}),
	\end{align*}
	where $\gamma$ is an ideal knot, i.e. 
	\begin{align*}
		\Delta[\gamma]^{-1}=\inf_{\mathcal{C}(\mathcal{K})}\Delta^{-1}.
	\end{align*}
\end{lemma}
\begin{proof}
	Using the H\"older inequality, it is easy to see the monotonically increasing convergence
	\begin{align*}
		\mathcal{M}_{s}^{\frac{1}{s}}(\gamma)=\norm*{\kappa(\gamma,\gamma,\gamma)}_{L^{s}([0,1]^{3})}
		\xrightarrow[s\to\infty]{}\norm*{\kappa(\gamma,\gamma,\gamma)}_{L^{\infty}([0,1]^{3})}=\Delta[\gamma]^{-1}.
	\end{align*}
	An application of Fatou's Theorem shows that the functionals are lower semi-continuous with regard to uniform convergence. 
	Hence, we immediately have $\Gamma$ convergence (see \cite[Remark 1.40 (ii), p.35]{Braides2002a}).
	For $s>3$, curves of finite energy are of class $C^{1}$ (see \cite{Strzelecki2010a}), so that we can restrict to this topology.
	Minimizers of the energies are known to exist and the monotonicity from above gives a uniform bound on, say, $\mathcal{M}_{4}(\gamma_{s})$.  
	These facts together with the $L^{\infty}$ bound imply $C^{1}$ subconvergence to a simple arc length curve of the same length which belongs to the same knot class (see \cite{Strzelecki2010a}).
	Now, the result is a consequence of Theorem \ref{dalmasoproposition}.
\end{proof}

\end{appendix}

\newcommand{\etalchar}[1]{$^{#1}$}
\providecommand{\bysame}{\leavevmode\hbox to3em{\hrulefill}\thinspace}

\end{document}